\documentclass[12pt]{amsart}

\usepackage[textsize=scriptsize]{todonotes}
\usepackage[sort,nocompress,noadjust]{cite}
\usepackage{color}
\usepackage{booktabs}
\usepackage{graphicx}
\usepackage{subcaption}
\usepackage{hyperref}

\usepackage{framed, mathtools}
\usepackage{hyperref}
\usepackage{fullpage}
\usepackage{setspace}
\usepackage{enumerate}
\newtheorem{theorem}{Theorem}
\newtheorem{lemma}[theorem]{Lemma}
\newtheorem{corollary}[theorem]{Corollary}
\usepackage{thmtools}

\declaretheoremstyle[notefont=\bfseries,notebraces={}{},%
    headpunct={},postheadspace=1em,spaceabove=0.5em,spacebelow=0.5em]{mystyle}
\declaretheorem[style=mystyle,numbered=no,name=Theorem]{thm-hand}
\declaretheorem[style=mystyle,numbered=no,name=Conjecture]{conj-hand}
\declaretheorem[style=mystyle,numbered=no,name=Definition.]{defn-hand}
\declaretheorem[style=mystyle,numbered=no,name=Theorem.]{thm-no}
\declaretheorem[style=mystyle,numbered=no,name=Conjecture.]{conj-no}
\declaretheorem[style=mystyle,numbered=no,name=Lemma.]{lemma-no}
\declaretheorem[style=mystyle,numbered=no,name=Lemma]{lemma-hand}
\makeatletter
\def\resetMathstrut@{%
  \setbox\z@\hbox{%
    \mathchardef\@tempa\mathcode`\[\relax
    \def\@tempb##1"##2##3{\the\textfont"##3\char"}%
    \expandafter\@tempb\meaning\@tempa \relax
  }%
  \ht\Mathstrutbox@\ht\z@ \dp\Mathstrutbox@\dp\z@}
\makeatother

\newtheorem{problem}[theorem]{Problem}
\newtheorem{question}[theorem]{Question}

\theoremstyle{definition}
\newtheorem{definition}[theorem]{Definition}

\usepackage{listings}
\theoremstyle{remark}

\allowdisplaybreaks

\title{On the Discrepancy Between Two Zagreb Indices}
\author{Ashwin Sah}
\thanks{Massachusetts Institute of Technology, Cambridge MA. Email: \texttt{asah@mit.edu}}
\date{\today}
\author{Mehtaab Sawhney}
\thanks{Massachusetts Institute of Technology, Cambridge MA. Email: \texttt{msawhney@mit.edu}}

\begin{document}

\maketitle
\begin{abstract}
We examine the quantity 
\[S(G) = \sum_{uv\in E(G)} \min(\deg u, \deg v)\]
over sets of graphs with a fixed number of edges.
The main result shows the maximum possible value of $S(G)$
is achieved by three different classes of constructions,
depending on the distance between the number of edges
and the nearest triangular number.
Furthermore we determine the maximum possible value
when the set of graphs is restricted to be bipartite, a forest,
or to be planar given sufficiently many edges.
The quantity $S(G)$ corresponds to the difference between two well studied indices,
the irregularity of a graph and the sum of the squares of the degrees in a graph.
These are known as the first and third Zagreb indices in the area of mathematical chemistry.
\end{abstract}
\section{Introduction}
\subsection{The specialty of a graph}
The following question appeared on the Team Selection Test for the 2018 United States
International Math Olympiad team.
\begin{problem}
At a university dinner, there are 2017 mathematicians who each order two distinct entr\'{e}es, with no two mathematicians ordering the same pair of entr\'{e}es. The cost of each entr\'{e}e is equal to the number of mathematicians who ordered it, and the university pays for each mathematician's less expensive entr\'{e}e (ties broken arbitrarily). Over all possible sets of orders, what is the maximum total amount the university could have paid?
\end{problem}
This problem, posed by Evan Chen, proved extremely challenging for contestants,
with only one full solution given on the contest.
We can rephrase the question in more graph theoretic terms.
\begin{definition}
Define the \textbf{specialty} of a graph $G$ to be
\[S(G) = \sum_{uv\in E(G)} \min(\deg u, \deg v)\]
where $E(G)$ is the edge set of a graph $G$.
\end{definition}
The question posed to the contestants therefore is equivalent to evaluating
\[F(2017)=\max_{G\text{ has 2017 edges}}S(G).\]
The given solutions relied heavily on the fact that $2017=\binom{64}{2}+1$, and therefore the maximizing graph is near a complete graph. The purpose of this note is to determine \[F(N)=\max_{G\text{ has $N$ edges}}S(G)\] 
in general,
as well as determine the maximum when $G$ is further restricted to be bipartite, a forest, or planar given sufficiently many edges in the final case. 

\subsection{Relation to Zagreb indices}
The specialty of a graph is intimately related to two quantities of a graph, the irregularity of a graph and the sum of the squares of the degrees.
First, Albertson \cite{albertson1997irregularity} defines the irregularity of $G$,
which we denote as $M_3(G)$, to be
\[  M_3(G) = \sum_{uv\in E(G)}
	\left\lvert \deg u - \deg v \right\rvert.\]
Fath-Tabar \cite{fath2011old} also defines this as the third Zagreb index,
hence the choice of notation.
Tavakoli and Gutman \cite{tavakoli2013extremely} as well as Abdo, Cohen, and Dimitrov \cite{abdo2012bounds} independently determined the maximum of $M_3(G)$ over all graphs with $n$ vertices.

On the other hand if the minimum of the degrees is replaced with a sum of the degrees in the definition of specialty, the corresponding quantity
\[M_1(G) = \sum_{uv\in E(G)} (\deg u + \deg v) = \sum_{v\in V(G)} (\deg v)^2\]
roughly counts the number of directed paths of length $2$ in $G$.
The problem of maximizing this quantity over all graphs with a particular number of edges and vertices was a problem introduced in 1971 by Katz \cite{katz1971rearrangements}. The first exact results in this problem were given by Ahlswede and Katona who in essence demonstrated that the maximum value is achieved on at least one of two possible graphs called the quasi-complete and quasi-star graphs \cite{ahlswede1978graphs}. However, as Erd{\H o}s remarked in his review of the paper, ``the solution is more difficult than one would expect'' \cite{Er}. \'{A}brego, Fern\'{a}ndez-Merchant, Neubauer, and Watkins furthered this result by determining the exact maximum in all cases \cite{abrego2009sum}. However, given the complexity of the exact value of the upper bound, there was considerable interest in giving suitable upper bounds and a vast literature of such bounds developed. See \cite{cheng2009extreme}, \cite{das2004maximizing}, \cite{das2015zagreb}, \cite{de1998upper}, \cite{nikiforov2007sum}, \cite{nikiforov2007sum}, \cite{zhou2006note}, \cite{zhou2007remarks} for many results of this type. Many of these results stem from the area of mathematical chemistry and the above quantity is referred to as the first Zagreb index, $M_1(G)$. In this context, using the notation in \cite{fath2011old}, we resolve the problem of maximizing
\begin{align*}
	S(G) &= \frac{1}{2}\sum_{uv\in E(G)} (\deg u + \deg v - 
		\left\lvert \deg u - \deg v \right\rvert) \\
	&= \frac{M_1(G) - M_3(G)}{2},
\end{align*}
that is, the discrepancy between two of these already-studied graph invariants, over graphs with a fixed number of edges. Note that both $M_1(G)$ and $M_3(G)$ can both trivially have order of the square of the number of edges, and in this paper we in fact show that $S(G)$ has a strictly lower order. Furthermore, the maximum of $S(G)$ being of lower order extends to when $G$ is restricted to be a bipartite graph, a forest, or a planar graph. (The maximum value of $M_1(G)$ over a fixed number of edges is achieved by a star \cite{ahlswede1978graphs}. For $M_3(G)$ the maximum value over the set of all trees is achieved by a star \cite{li2016irregularity} and one can easily check this extends to all planar graphs.)

\subsection{Combinatorial interpretation}
We end with an alternate combinatorial interpretation of $S(G)$
arising through the related $S'(G)$ where
\[ S'(G) = \frac{1}{3} \sum_{uv\in E(G)} (\min(\deg u, \deg v)-1)
	= \frac13 S(G) - \#E(G) .\]
Note that $S'(G)$ provides a trivial upper bound for the number of triangles in a graph $G$ and a solution to the initial problem therefore provides an upper bound for the number of triangles in a graph with a specified number of edges.

Erd\H{o}s gave a remarkably short proof that for graphs with $N=\binom{n}{2}+m$ edges
(with $1\le m\le n$), the maximum number of triangles
is achieved on a complete graph with $n$ vertices
and an additional vertex connected to $m$ vertices in the clique \cite{erdos1962number}.
The remarkable fact therefore is that the maximum of $S(G)$
is not always achieved on the same graphs
as those that maximize the number of triangles,
despite the optimal constructions agreeing for infinitely many integers
(with a density of $\frac{2}{5}$).

\section{Maximum Specialty over all Graphs}
We will show the following result, which determines $F(N)$ in general.
\begin{theorem}\label{main}
Represent $N = \binom{n}{2} + m, 1\le m\le n$ uniquely.
Then the maximum value of $S(G)$ on a graph on $N$ edges is attained on a graph $G$ (which is not necessarily unique):
\begin{enumerate}
\item[i.] If $1\le m\le\frac{2n - 3}{5}$, then $G$ is a clique of size $n$ and an additional vertex $v$ which connects to $m$ vertices in the clique. Then $F(N)=(n-1)\binom{n}{2}+\frac{m(3m-1)}{2}$ in this case.
\item[ii.] If $\frac{2n - 3}{5}\le m\le\frac{n - 1}{2}$, then $G$ consists of three parts. 

\begin{itemize}
    \item A clique of size $2m$ missing $m$ disjoint edges
    \item A clique of size $n - 2m$ with every vertex in this clique connected to every vertex in the previous ``almost-clique"
    \item A single vertex connected to every vertex in the ``almost-clique" of size $2m$ but to no vertices in the clique of size $n - 2m$
\end{itemize}
In this case $F(N)=(n - 1)\binom{n}{2} + m(4m - n + 1)$.
\item[iii.] If $\frac{n - 1}{2} < m\le n$, then $G$ is a ${K}_{n + 1}$ missing $n - m$ disjoint edges in the clique. In this case $F(N)=n\binom{n}{2} + m(2m - n).$
\end{enumerate}
\end{theorem}
\begin{figure}[h]
\begin{subfigure}[b]{.3\linewidth}
\begin{center}
  \includegraphics[width=.73\linewidth]{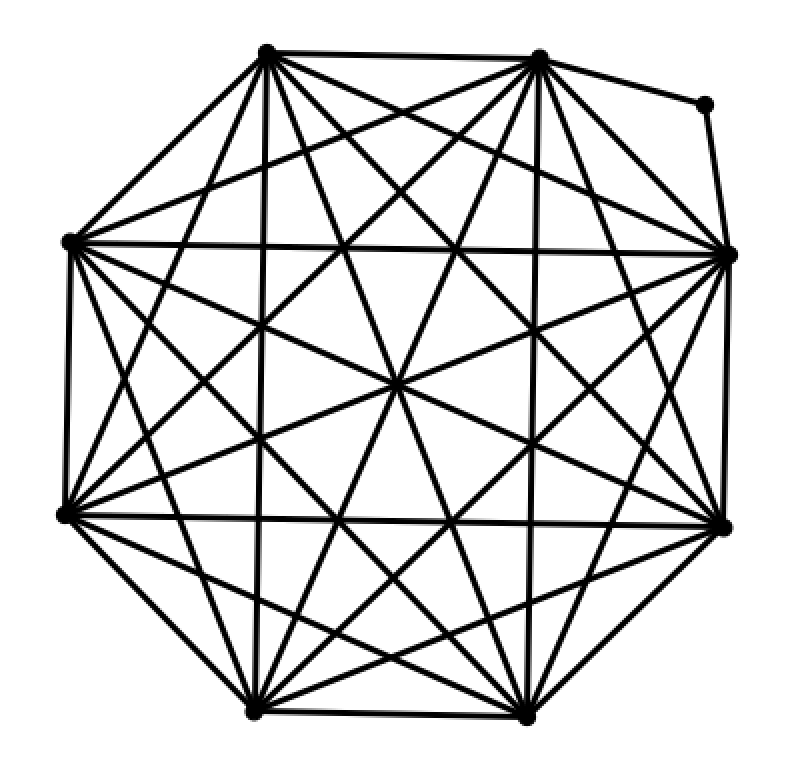}
  \caption{$N=30$ edges}
\end{center}
\end{subfigure}
\hfill
\begin{subfigure}[b]{.3\linewidth}
\begin{center}
  \includegraphics[width=.80\linewidth]{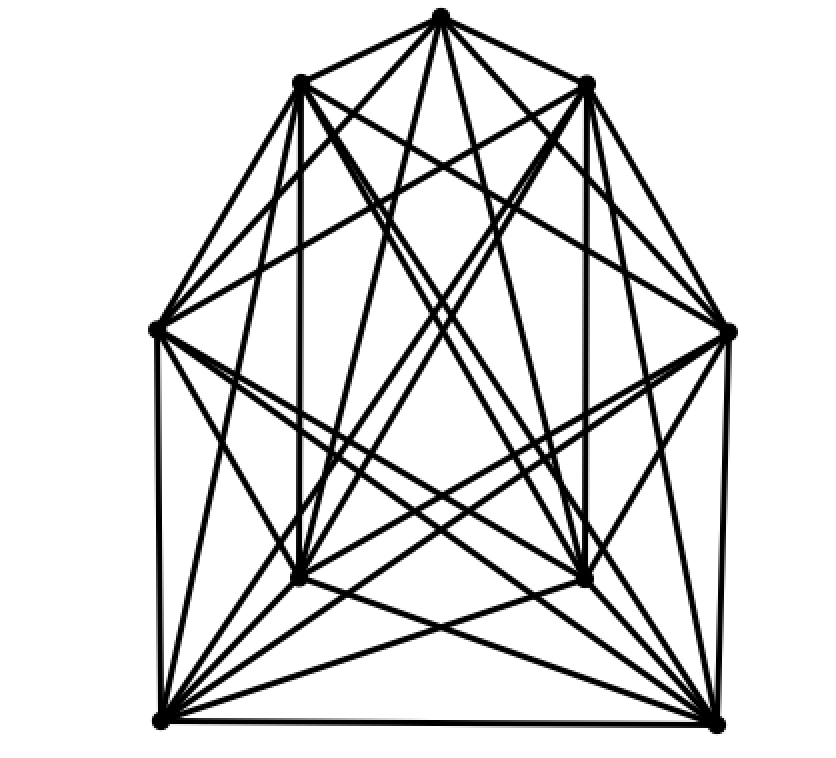}
  \caption{$N=31$ edges}
\end{center}
\end{subfigure}
\hfill
\begin{subfigure}[b]{.3\linewidth}
\begin{center}
  \includegraphics[width=.73\linewidth]{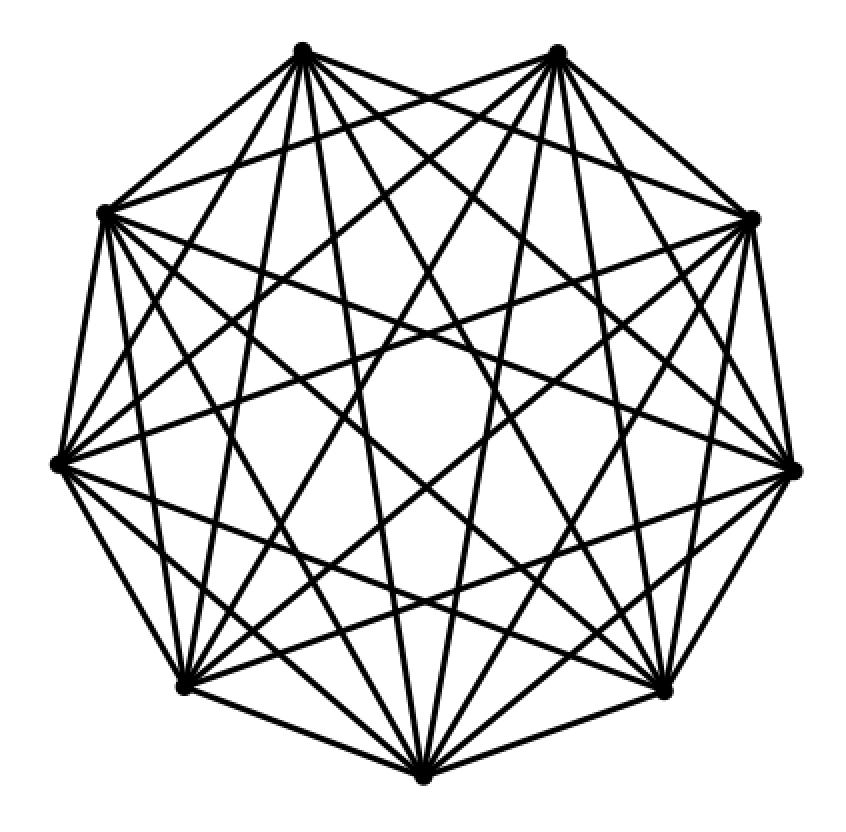}
  \caption{$N=32$ edges}
 \end{center}
\end{subfigure}
\caption{Maximal graphs in Theorem \ref{main}}
\label{fig:arckdef1}
\end{figure}
There are two key aspects to the claimed maximal graph $G$. First, in each case $G$ has $n + 1$ vertices. This is no coincidence, and is a key structural result in the course of proving Theorem \ref{main}. Secondly, these maximal graphs contain a ``universal'' vertex connected to all other vertices in both the first and third cases, but not in the second case. The analysis in the following sections is therefore often separated based on whether or not the graph contains such a ``universal'' vertex. As it turns out, in the case where the graph has no universal vertex and has $n + 1$ vertices, we will show the construction in (ii) is optimal for all $1\le m\le\frac{n - 1}{2}$.

Before proceeding with the bulk of the proof we need a series of definitions.
\begin{definition}
In a graph $G$, define a vertex $v$ to be \textbf{universal} if $v$ connects to all other vertices in the graph $G$. Furthermore, let the set of graphs $G$ with a universal vertex be $\mathcal{UG}$.
\end{definition}
\begin{definition}
For an edge $uv\in E(G)$, define its \textbf{weight} to be $\min(\deg u,\deg v)$.
\end{definition}
\begin{definition}
Let
\[C(N) = \max_{\substack{G\text{ has $N$ edges}\\G\text{ has $n+1$ vertices} \\ G\notin\mathcal{UG}}} (S(G)).\]

We leave $C(N)$ undefined if no such graph $G$ exists.
\end{definition}
For convenience we consider $N\le 2$ separately. Note that $F(1) = 2, F(2) = 2$ as there is only one possible value in both cases. Since $1 = \binom{1}{2} + 1$ and $2 = \binom{2}{2} + 1$, these both agree with the claimed formula $n\binom{n}{2} + m(2m - n)$ in Theorem \ref{main}. 

We now note that $C(N)$ is only defined if $1\le m\le\frac{n - 1}{2}$.
\begin{lemma}\label{well-defined-m}
If $m > \frac{n - 1}{2}$ then every graph with $n + 1$ vertices and $N$ edges is in $\mathcal{UG}$.
\end{lemma}
\begin{proof}
We have $n + 1$ vertices but $\binom{n}{2} + m = \binom{n+1}{2} - (n - m)$ edges. Since $2(n - m) < n+1$ and each edge can make at most $2$ vertices non-universal, there must be an universal vertex.
\end{proof}
Furthermore note that $F(N)$ is monotonically increasing as $N$ increases.
\begin{lemma}\label{monotonicity}
$F(N)$ is a strictly increasing function with respect to $N$.
\end{lemma}
\begin{proof}
Consider graph $G$ with $S(G) = F(N)$. Make $G'$ be $G$ with additional vertex $v$ connected to an arbitrary vertex in $G$. Then $F(N) = S(G) < S(G')\le F(N + 1)$.
\end{proof}
The next observation was the key observation necessary for the original problem given to students on the Team Selection Test. 
\begin{lemma}\label{univ or n+1}
The maximum $F(N)$ is attained either on a graph with $n + 1$ vertices or a graph with a universal vertex.
\end{lemma}
\begin{proof}
Consider a graph $G$ with $n'\ge n + 2$ vertices and no universal vertex. Let $v$ be the vertex with minimal degree $\ell\le n' - 2$ and suppose the neighbors of $v$ are $v_1, \ldots, v_{\ell}$. Since there is no universal vertex in $G$, each of $v_1, \ldots, v_{\ell}$ has a vertex $w_1, \ldots, w_{\ell}$ with $w_iv_i$ not being an edge for each $1\le i\le \ell$. 

Now delete all edges $vv_i$ in $G$ and replace these edges with $v_iw_i$ and delete the vertex $v$. Call this new multigraph $G'$. Note that $G'$ has $n' - 1$ vertices and that multiple edges may arise in $G'$ if and only if $v_i = w_j$ and $w_j = v_i$. Construct $G''$ by taking any pair of double edges, deleting one of them, and adding any missing edge of $G'$ in its place. This is always possible since $N = \binom{n}{2} + m\le \binom{n+1}{2}\le \binom{n'-1}{2}$.

Note that $G''$ has $n' - 1$ vertices and $S(G'')\ge S(G)$. The second observation follows as every vertex in $G''$ has degree at least as large as in $G$, while the $\ell$ edges deleted from $G$ have been replaced with $\ell$ new edges with increased or the same weights. Iterating this procedure, we eventually terminate since the vertex count decreases every time. Furthermore, we terminate at a graph that either has a universal vertex or has $n + 1$ vertices, with at least as large specialty as before, which implies the result.
\end{proof}
Surprisingly, one can leverage this observation to reduce the search of graphs which maximize $F(N)$ to only those on $n+1$ vertices.
\begin{lemma}\label{n+1}
The maximum $F(N)$ is attained on a graph with $n + 1$ vertices.
\end{lemma}
\begin{proof}
We induct on $N$. The cases when $N = 1$ or $N=2$ are trivial so let $N\ge 3$ for the remainder of the proof. Suppose that the result holds for all smaller $N$ and set $N = \binom{n}{2} + m, 1\le m\le n$. Note that $n\ge 2$ as $N\ge 3$. Now suppose for the sake of contradiction that the maximum is not attained on a graph with $n + 1$ vertices. Therefore by Lemma \ref{univ or n+1}, we know that there exists a graph $G$ with a universal vertex satisfying $S(G) = F(N)$ but no such graph with $n + 1$ vertices. Therefore $G$ has $n'\ge n + 2$ vertices and $G$ has a universal vertex $v$. Label the neighbors of $v$ as $v_1, \ldots, v_{n'-1}$. Furthermore, let vertex $v_i$ have degree $d_i$.

Consider deleting $v$ from $G$. The remaining graph, $G'$, has $N - n' + 1$ edges and the remaining vertices have degree $1$ less in $G'$ than in $G$. Therefore, each of the remaining $N - n' + 1$ edge weights decrease by $1$ when going from $S(G)$ to $S(G')$. Furthermore, the $n' - 1$ edges $vv_i$ have weight $d_i$ in $G$. Therefore the total loss from removing these edges is
\[\sum_{i = 1}^{n'-1} d_i = 2N - n' + 1.\]
Thus
\begin{align*}
S(G) &= (2N - n' + 1) + (N - n' + 1) + S(G')\\
&\le 3N - 2n' + 2 + F(N - n' + 1)\\
&< 3N - 2n + F(N - n),
\end{align*}
where we have used $n'\ge n + 2$ in the final inequality. 

However, consider $G''$ with $N - n = \binom{n - 1}{2} + (m - 1)$ edges that has $S(G'')=F(N-n)$. If $2\le m\le n$ then $1\le m - 1\le n - 1$ so by the inductive hypothesis $G''$ can be taken to have $n$ vertices. If $m = 1$, then $N - n = {n - 2\choose 2} + (n - 1) = \binom{n}{2}$ and by the inductive hypothesis $F(N)$ is maximized on a graph with $n - 1$ vertices. In this case, add an empty vertex to obtain $G''$. Let the vertices of $G''$ be $v_1,\ldots,v_n$ and let $v_i$ have degree $d_i$.
Now add a universal vertex $v$ to $G''$ to form graph $G^{\circ}$ with $(N - n) + n = N$ edges. The  weights of all edges in $G''$ increase by $1$ and inserted edges $vv_i$ have weight $d_i + 1$. Therefore
\begin{align*}
S(G^{\circ}) &= S(G'') + (N - n) + n + 2(N - n)\\
&= S(G'') + 3N - 2n\\
&= F(N - n) + 3N - 2n,
\end{align*}
and we have constructed a graph on $N$ edges with $S(G^{\circ}) > S(G)=F(N)$, a contradiction! Thus the inductive step is complete and the result follows.
\end{proof}
With this structural result one can already deduce that the specialty of graphs with a triangular number of edges is maximized with a complete graph.
\begin{corollary}\label{m=n}
If $m = n$, then $F(N) = n{n + 1\choose 2}$.
\end{corollary}
\begin{proof}
Note that $N = \binom{n}{2} + n = {n + 1\choose 2}$ in this case and that the maximal value $F(N)$ is attained on a graph with $n + 1$ vertices by Lemma \ref{n+1}. Therefore the complete graph $K_{n+1}$ is the only possibility and the result follows.
\end{proof}
Furthermore, we can now derive an inductive relationship between $F(N)$ and $C(N)$.
\begin{lemma}\label{inductive-max}
$F(N) = \max(F(N - n) + 3N - 2n, C(N))$
\end{lemma}
\begin{proof}
By Lemma \ref{n+1} we know $F(N)$ is realized on a graph $G$ with $n + 1$ vertices. If $G$ has no universal vertex then $F(N) = C(N)$. Otherwise $G$ has $n + 1$ vertices and a universal vertex. We show that any such graph has specialty at most $F(N - n) + 3N - 2n$, and furthermore that there is a construction to achieve this bound.
We now prove that $F(N)\le F(N - n) + 3N - 2n$ if the graph $G$ has a universal vertex $v$. Let $v$ have neighbors $v_1, \ldots, v_n$ as $G$ and let vertex $v_i$ have degree $d_i$. The edges $vv_i$ have weight $d_i$, for a total of $2N - n$. If we construct $G'$ by removing $v$ from $G$, every remaining edge has decreased in weight by $1$. Therefore we have
\[S(G') = S(G) - (N - n) - (2N - n) = F(N) - 3N + 2n.\]
But by the definition of $F$, we have $S(G')\le F(N - n)$, since $G'$ has $N - n $ edges. 
Therefore it follows in this case that 
\[F(N)\le F(N - n) + 3N - 2n.\]
This bound can be achieved by taking a graph on $n$ vertices and $N-n$ edges with specialty $F(N-n)$ and adding a vertex that connects to every other vertex. An isolated vertex need be first added in the case $m = 1$. The analysis mimics the previous paragraph, and this is in essence the same as the construction in Lemma \ref{n+1}.

Therefore, since either the optimum $G$ with $n + 1$ vertices has a universal vertex or not, $F(N) = \max(F(N - n) + 3N - 2n, C(N))$ is forced to hold and the result follows.
\end{proof}
\begin{corollary}\label{inductive-above-half} If $m > \frac{n - 1}{2}$ then $F(N) = F(N - n) + 3N - 2n$.
\end{corollary}
\begin{proof}
Note that by Lemma \ref{well-defined-m}, the quantity $C(N)$ is not well-defined in this range. The argument of Lemma \ref{inductive-max} goes through without change except $G$ is forced to have a universal vertex in this case.
\end{proof}

The final structural result we use relies on the key idea of the proof of the Havel-Hakimi algorithm \cite{hakimi1962realizability} \cite{havel1955remark}, which controls possible degree sequences of a simple graph. (A degree sequence of a graph is the list of degrees of its vertices in some order.)
\begin{lemma}\label{Havel-Hakimi Lemma}
Consider a graph $G$ with a weakly decreasing degree sequence $(d_1,\ldots,d_k)$ satisfying $d_1=k-2$ and $k\ge 2$. Then there exists a graph $G'$ such that $G'$ has the same degree sequence as $G$, $S(G')\ge S(G)$, and one vertex of degree $d_1$ in $G'$ is connected to all vertices except a vertex of minimal degree.  
\end{lemma}
\begin{proof}
Label the vertices of $G$ as $v_1,\ldots,v_k$ with the degree of $v_i$ being $d_i$. Now the neighborhood of $v_1$ is missing a unique vertex $v_j$. If $j = k$ then taking $G = G'$ gives the result. Otherwise $j\neq k$ and note that $v_j$ has a neighborhood at least as large as $v_k$. Since $v_1$ is connected to $v_k$ but not $v_j$, there exists $v_{\ell}$ such that $(v_{\ell},v_j)$ is an edge but $(v_{\ell},v_k)$ is not. Then define $G'$ by adding in $(v_{\ell},v_k)$ and $(v_1,v_j)$ and removing $(v_1,v_k)$ and $(v_{\ell},v_j)$. Note that
$S(G') - S(G) = d_j + d_k - d_k - \min(d_{\ell}, d_j)\ge 0$, and every vertex in $G'$ has the same degree in $G$. The result follows.
\end{proof}
We now give the main technical lemma in this section of the paper. In particular we recursively bound the specialty of all graphs without a universal vertex.
\begin{lemma}\label{Nonuniversal Recurence}
Suppose that $N=\binom{n}{2}+m$ with $m < \frac{n - 1}{2}$. Then it follows that $C(N)\le\frac{(n - 1)(3n - 4)}{2} - m + C(N - (n - 1))$.
\end{lemma}
\begin{proof}
Note that $N-(n-1)=\binom{n-1}{2}+m$ and therefore since $n-2m>1$ it follows that $(n-1)-2m\ge 1$. Thus the right hand side of the claimed inequality is well defined. Now consider $G$ with $n + 1$ vertices, no universal vertex, and such that $S(G)=C(N)$. Furthermore note that there is a vertex of degree $n-1$ as the average degree is $\ge\frac{n(n-1)}{n+1}>n-2$ and there is no universal vertex in $G$.

Now let $G$ have vertices $v_1,\ldots,v_{n+1}$ with $\deg v_i = d_i$, and where $d_i$ is a nondecreasing sequence. By Lemma \ref{Havel-Hakimi Lemma} we can further assume that in $G$, $v_1$ connects to $v_2$ through $v_n$ but not $v_{n+1}$. Furthermore, note that $d_{n+1}$ is not $n-1$ as otherwise $\sum_{i=1}^{n+1}d_i=(n+1)(n-1)=2\left(\binom{n}{2}+m\right)$ which implies $m=\frac{n-1}{2}$, a contradiction.

Suppose that $v_{n+1}$ is connected to $\ell$ vertices of degree $d_{n+1}$. Then we claim that 
\[S(G) = (2N - (n - 1) - d_{n+1}) + (N - (n - 1) - d_{n+1}) + \ell + S(G'),\]
where $G'$ is induced subgraph of $G$ on $v_2,\ldots,v_{n+1}$. This follows as
\begin{align*}
S(G) &= \sum_{(v_i,v_j)\in E(G)}\min(d_i,d_j)\\
&= \sum_{\substack{(v_i,v_j)\in E(G)\\ i,j\neq n+1}}\min(d_i,d_j)+d_{n+1}^2\\
&= \sum_{\substack{(v_i,v_j)\in E(G)\\ i,j\neq 1,n+1}}\min(d_i,d_j)+\sum_{i\neq 1,{n+1}}d_i+d_{n+1}^2\\
&= \sum_{\substack{(v_i,v_j)\in E(G)\\ i,j\neq 1,n+1}}(\min(d_i,d_j)-1) + (d_{n + 1}^2 - \ell) \\
&+ (3N-2d_{n+1}-2(n-1)) + \ell\\
&= S(G')+3N-2d_{n+1}-2(n-1)+\ell.
\end{align*}
The last step follows due to a few facts about the removal of $v_1$. Every edge $(v_i, v_j)$ in $E(G')$ for $i, j\neq 1, n + 1$ is has weight $1$ less than it does in $G$. The edges attached to $v_1$ are all removed. The edges attached to $v_{n + 1}$ are between a degree $d_j$ and $d_{n + 1}$ vertex, which has weight $d_{n + 1}$ in $G$. In $G'$, the degrees are $d_j - 1, d_{n + 1}$. If $d_j > d_{n + 1}$ the weight is still $d_{n + 1}$ in $G'$, but if $d_j = d_{n + 1}$ then the weight has been decremented by $1$. This happens to precisely $\ell$ edges, by definition, hence the claimed equality.

As $d_{n+1}<n-1$ and the remaining vertices have decreased degree by $1$, it follows that $G'$ does not have a universal vertex and we find $S(G')\le C(N-(n-1))$. Therefore it follows that $S(G)\le C(N-(n-1))+3N-2d_{n+1}-2(n-1)+\ell$.

The key claim is now that $\ell - 2d_{n+1}\le -4m$. This yields
\begin{align*}
	C(N) &\le (3N-4m-2(n-1)) + C(N - (n - 1)) \\
	&=\frac{(n - 1)(3n - 4)}{2} - m + C(N - (n - 1)),
\end{align*}
 as desired.

Now suppose that $\ell > 2d_{n+1} - 4m$. Since the sum of all degrees in $G$ is $n^2 - n + 2m$ and then $n^2-n+2m=\sum_{i=1}^{n+1}d_i\le (n)(n-1)+d_{n+1}$, we conclude $d_{n+1}\ge 2m$. Further note from earlier that $d_{n+1}\le n-2$. Therefore it follows that there are at least $\ell + 1 \ge 2d_{n+1} - 4m + 2$ vertices of degree $d_{n+1}$ and thus
\begin{align*}
	n^2-n+2m =\sum_{i=1}^{n+1}d_i &\le d_{n + 1}(2d_{n + 1} - 4m + 2) \\
	&+ (n - 1)(n + 1 - (2d_{n + 1} - 4m + 2)).
\end{align*}
But note that the rightmost expression is a convex function in $d_{n+1}$. Thus, its maximum possible value over $d_{n + 1}\in [2m, n - 2]$ is attained at an endpoint. But its values at $d_{n+1} = 2m$ and $d_{n+1} = n - 2$ both equal $n^2 - n + 2m + (2m - n + 1)$. Since $m < \frac{n - 1}{2}$, this is strictly less than $n^2 - n + 2m$, a contradiction.
\end{proof}

Using this lemma we can now calculate $C(N)$ explicitly.

\begin{lemma}\label{non-univ-max}
For $1\le m\le\frac{n - 1}{2}$,
\[C(N) = (n - 1)\binom{n}{2} + m(4m - n + 1).\]
\end{lemma}
\begin{proof}
The construction given in part (ii) of Theorem \ref{main} is valid in the given range and has no universal vertex for any $1\le m\le\frac{n - 1}{2}$ while achieving the claimed bound.

Now the key point is that when $m = \frac{n-1}{2}$ there is a single isomorphism class of graphs that we are maximizing over for $C(N)$, a $K_{n+1}$ missing a perfect matching. In this case every edge has weight $n-1$ and the result follows.

Otherwise, since $N=\binom{n}{2}+m$, applying Lemma \ref{Nonuniversal Recurence} $n-2m-1$ times inductively yields
\[C(N)\le \sum_{i=2m+2}^{n}\left(\frac{(i-1)(3i-4)}{2}-m\right)+C\left(\binom{2m+1}{2}+m\right)\]
\[= (n-1)\binom{n}{2} + m(4m-n+1)\]
since ${2m + 1\choose 2} + m$ is precisely in the $m = \frac{n - 1}{2}$ case already discussed. The result follows.
\end{proof}
\begin{lemma}\label{final}
We claim that
\begin{itemize}
    \item If $1\le m\le \frac{2n-3}{5}$ then $F(N)\le (n-1)\binom{n}{2}+\frac{m(3m-1)}{2}$.
    \item If $\frac{2n-3}{5}\le m\le \frac{n-1}{2}$ then $F(N)\le (n-1)\binom{n}{2}+m(4m-n+1)$.
    \item If $\frac{n-1}{2}< m\le n$ then $F(N)\le n\binom{n}{2}+m(2m-n)$.
\end{itemize}
\end{lemma}
With Lemma \ref{final} and the constructions given in Theorem \ref{main} the main result follows.
\begin{proof}
We prove this by induction on $N$. The result is trivial for $N=1$ and $N=2$. Furthermore $N=4$ and $N=5$ follow from a direct verification and $N = 3, 6$ follow from Lemma \ref{m=n}. Now suppose we have proved the claim for all $N<\binom{n}{2}+m$ and now consider $N=\binom{n}{2}+m$. Note that since $N\ge 7$ it follows that $n\ge 4$ in the remaining analysis. Furthermore note that $N-n=\binom{n-1}{2}+{m-1}$, which is used throughout in the below analysis. We now consider cases based on the relative size of $m$ and $n$.

Case 1: If $m=1$ then by Lemma \ref{inductive-max} and Lemma \ref{non-univ-max} if follows that
\begin{align*}
F(N) &=\max\bigg(F\bigg(\binom{n-1}{2}\bigg)+3N-2n,(n-1)\binom{n}{2}+(5-n)\bigg) \\
&= \max\bigg((n-2)\binom{n-1}{2}+3\binom{n}{2}+3-2n,(n-1)\binom{n}{2}+(5-n)\bigg) \\
&= (n-1)\binom{n}{2}+1
\end{align*}
where we used Lemma \ref{m=n} in the second step and $n\ge 4$ in the final step.

Case 2: If $2\le m\le \frac{2n-3}{5}$ then by Lemma \ref{inductive-max} and Lemma \ref{non-univ-max} we find
\[F(N)=\max\bigg(F\bigg(\binom{n-1}{2}+m-1\bigg)+3N-2m,(n-1)\binom{n}{2}+m(4m-n+1)\bigg).\] Since $1\le m-1\le \frac{2(n-1)-3}{5}$ it follows that
\[F(N-n)+3N-2n=(n-2)\binom{n-1}{2}+\frac{(m-1)(3m-5)}{2}+3N-2n\] 
\[=(n-1)\binom{n}{2}+\frac{m(3m-1)}{2}\]
and note that if $m\le \frac{2n-3}{5}$ then $\frac{m(3m-1)}{2}\ge m(4m-n+1)$. Therefore
\[F(N)=\max\bigg((n-1)\binom{n}{2}+\frac{m(3m-1)}{2},(n-1)\binom{n}{2}+m(4m-n+1)\bigg)
\]
\[=(n-1)\binom{n}{2}+\frac{m(3m-1)}{2},\] as desired.

Case 3: If $\frac{2n-3}{5}<m\le\frac{2n}{5}$ then note that $1\le m-1\le \frac{2(n-1)-3}{5}$. Then
\[F(N-n)+3N-2n=(n-2)\binom{n-1}{2}+\frac{(m-1)(3m-5)}{2}+3N-2n\] 
\[=(n-1)\binom{n}{2}+\frac{m(3m-1)}{2}.\]
Since $m(4m-n+1)\ge \frac{m(3m-1)}{2}$ in this range Lemma \ref{inductive-max} implies 
\[F(N)=\max\bigg((n-1)\binom{n}{2}+\frac{m(3m-1)}{2},(n-1)\binom{n}{2}+m(4m-n+1)\bigg)
\]
\[=(n-1)\binom{n}{2}+m(4m-n+1),\] as desired.

Case 4: If $\frac{2n}{5}<m\le\frac{n-1}{2}$ then note that $\frac{2(n-1)-3}{5}\le m-1\le \frac{(n-1)-1}{2}$. Therefore we compute
\[F(N-n)+3N-2n=(n-2)\binom{n-1}{2}+(m-1)(4m-n-2)+3N-2n\] 
\[=(n-1)\binom{n}{2}+m(4m-n-3)+n.\]
Since $m(4m-n+1)\ge m(4m-n-3)+n$ in this range it follows by Lemma \ref{inductive-max} that 
\[F(N)=\max\bigg((n-1)\binom{n}{2}+m(4m-n-3)+n,(n-1)\binom{n}{2}+m(4m-n+1)\bigg)
\]
\[=(n-1)\binom{n}{2}+m(4m-n+1),\] as desired.

Case 5: If $\frac{n-1}{2}<m< \frac{n+1}{2}$ then it follows that $m=\frac{n}{2}$. Then note that $\frac{2(n-1)-3}{5}\le m-1=\frac{(n-1)-1}{2}$ and $m-1\ge 1$. Therefore using Corollary \ref{inductive-above-half} we obtain
\[F(N)=F(N-n)+3N-2n\]
\[=(n-1)\binom{n}{2}+m(4m-n-3)+n\]
\[=(n)\binom{n}{2}+m(2m-n),\]
where $m=\frac{n}{2}$ is used in the final step.

Case 6: If $\frac{n+1}{2}\le m\le n$ then note that $\frac{(n-1)}{2}\le m-1\le n-1$. Thus
\[F(N)=F(N-n)+3N-2n\]
\[=(n-1)\binom{n-1}{2}+(m-1)(2m-n-1)+3\binom{n}{2}+3m-2n\]
\[=n\binom{n}{2}+m(2m-n),\]
as claimed.

Hence the result follows in all cases by induction.
\end{proof}
\section{Maximum Specialty over Bipartite Graphs}
In this section we compute 
\[F_{B}(N)=\max_{\substack{G\text{ has $N$ edges}\\G\text{ is bipartite}}} (S(G)).\]
In particular we prove the following theorem. 
\begin{theorem}\label{main2}
Suppose that $N=n^2+m$ for $1\le m\le 2n+1$, this decomposition is unique for $N\ge 1$. Then we have two cases based on size of $m$.
\begin{itemize}
    \item If $1\le m\le n$ then $F_B(N)=n^3+m^2$. This is achieved by taking a $K_{n,n}$ and an additional vertex that connects to $m$ vertices on one side of the original bipartition.
    \item If $n+1\le m\le 2n+1$ then $F_B(N)=n^3+n^2+m(m-n)$. This is achieved by taking a $K_{n+1,n+1}$ and removing $2n+1-m$ disjoint edges.
\end{itemize}
\end{theorem}
\begin{figure}[h]
\begin{subfigure}[b]{.49\linewidth}
    \begin{center}
  \includegraphics[width=.45\linewidth]{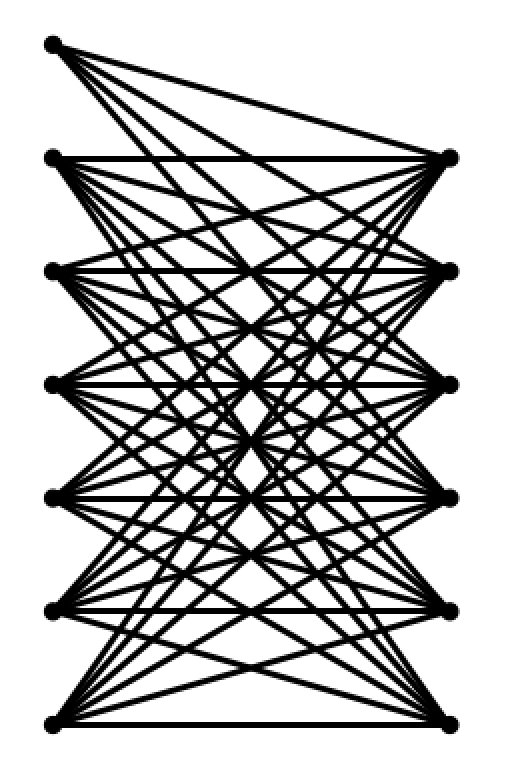}
  \caption{$N=40$ edges}
  \end{center}
\end{subfigure}
\hfill
\begin{subfigure}[b]{.40\linewidth}
\begin{center}
  \includegraphics[width=.5\linewidth]{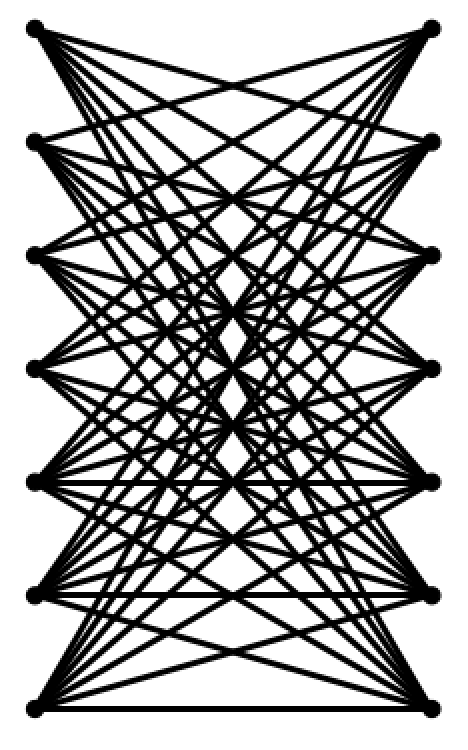}
  \caption{$N=45$ edges}
  \end{center}
\end{subfigure}
\caption{Maximal Graphs in Theorem \ref{main2}}
\label{fig:arckdef2}
\end{figure}

Note that $F_B(N)$ is trivially increasing; the proof is nearly identical to that of Lemma \ref{monotonicity}. This will be used throughout. The key to this section lies in an analog of Lemma $\ref{univ or n+1}$, but in this case the proof gives a stronger conclusion. In a bipartite graph, a \textbf{maximal} vertex of one side of the bipartition is a vertex which connects to all the vertices in the other side.
\begin{lemma}\label{maximal2}
For every integer $N$ there exists a graph $G$ with a bipartition $W$ and $X$ such that each partition has a maximal vertex and $S(G)=F_B(N)$.
\end{lemma}
\begin{proof}
Suppose that $G$ has $N$ edges with $S(G)=F_B(N)$ but has at least one partition that does not have a vertex connected to all other vertices in the other partition. Let the bipartition of $G$ be $W$ and $X$, and suppose that $W=\{w_1,\ldots,w_k\}$ and $X=\{x_1,\ldots,x_\ell\}$. If $X$ does not have a universal vertex, then consider the vertex in $W$ of minimal degree, and without loss of generality let this be $w_k$. For each vertex $x_i$ adjacent to $w_k$ there exists an alternate vertex $w_{f(i)}$ in $W$ to which $x_i$ does not connect, by assumption. Replace the edges $w_kx_i$ with $w_{f(i)}x_i$, and remove $w_k$. Every remaining vertex has at least the same degree as it did before, and the replaced edges have at least the same weight. If $G'$ is the altered graph, then $S(G')\ge S(G) = F_B(N)$ and $G'$ still has $N$ edges, hence $S(G') = F_B(N)$. We can iterate this process, which decreases the vertex count each time. Thus it terminates, and it must terminate when both sides of the bipartition have a maximal vertex, as desired.
\end{proof}
With Lemma \ref{maximal2} it is now possible to derive two relations regarding $F_B(N)$.
\begin{lemma}\label{recur}
If $N=n^2+m$ with $n+1\le m\le 2n+1$ then 
\[F_B(N)\le F_B(N-2n-1)+3(N-n)-2.\] Otherwise $N=n^2+m$ with $1\le m\le n$ and \[F_B(N)\le \max\bigg(F_B(N-2n-1)+3(N-n)-2,F(N-2n)+2N+n(n-3)\bigg).\]
\end{lemma}
\begin{proof}
Consider a graph $G$ with bipartition $W=\{w_1,\ldots,w_k\}$ and $X=\{x_1,\ldots,x_\ell\}$ satisfying $k\ge \ell$. Furthermore define $w_i$ to have degree $y_i$ and $x_j$ to have degree $z_j$. Applying Lemma \ref{maximal2}, we may assume that $y_k=\ell$ and $z_\ell=k$. Then consider $G'$ when one removes $w_k$ and $x_\ell$. We find
\[F_B(N)=(S(G')+N+1-k-\ell)+\sum_{j=2}^{\ell}\min(y_j,\ell)+\sum_{j=2}^{k}\min(z_j,k)+\ell.\]
If $\ell\ge n+1$ then $k + l\ge 2n + 2$ so $G'$ having $N + 1 - k - l$ edges implies 
\begin{align*}
F_B(N) &\le S(G')+N+1-k-\ell-2(N-\ell)+\ell\\
&\le F(N-2n-1)+3(N-n)-2.
\end{align*}
Otherwise suppose that $n+1\le m\le 2n+1$ and $\ell\le n$. Then we find $k+\ell\ge 2\sqrt{k\ell}\ge 2\sqrt{N} > 2n + 1$ so $k + \ell\ge 2n + 2$ and
\begin{align*}
F_B(N) &\le (S(G')+N+1-k-\ell)+\ell(\ell-1)+(N-\ell)+\ell\\
&\le F_B(N-2n-1)+2N-2n-1+n(n-1)\\
&\le F_B(N-2n-1)+3N-3n-2.
\end{align*}
Finally, suppose that $1\le m\le n$ and $\ell\le n$. Then $k+\ell\ge 2\sqrt{k\ell}\ge 2\sqrt{N} > 2n$ so $k + \ell\ge 2n+1$ and we find 
\begin{align*}
F_B(N) &\le (S(G')+N+1-k-\ell)+\ell(\ell-1)+(N-\ell)+\ell\\
&\le F_B(N-2n)+2N+n(n-3),
\end{align*}
as desired.
\end{proof}
We are now in a position to prove Theorem \ref{main2}.
\begin{lemma}
Let $N=n^2+m$ with $1\le m\le 2n+1$. Then we have two cases.
\begin{itemize}
    \item If $1\le m\le n$ then $F_B(N)\le n^3+m^2$.
    \item If $n+1\le m\le 2n+1$ then $F_B(N)\le n^3+n^2+m(m-n).$
\end{itemize}
\end{lemma}
This gives Theorem \ref{main2} once one observes equality can be attained.
\begin{proof}
The proof proceeds by a direct induction on $N$. Note that the result is trivial for $N=1$, $N=2$ and can easily be verified for $N=3$ and $N=4$. Therefore we will assume that we are considering $N\ge 5$ and therefore $n\ge 2$ in the below analysis. 

Case 1: If $m=1$ then Lemma \ref{recur} gives that 
\begin{align*}
F_B(N)&\le\max(F_B((n-1)^2-1)+3N-3n-2,F_B((n-1)^2)+2N+n(n-3))\\
&\le\max(n^3-3n+5,n^3+1)=n^3+1,
\end{align*}
as desired.

Case 2: If $m=2$ then Lemma \ref{recur} gives 
\begin{align*}
F_B(N)&\le \max(F_B((n-1)^2)+3N-3n-2,F_B((n-1)^2+1)+2N+n(n-3))\\
&\le\max(n^3+3,n^3+4)=n^3+4,
\end{align*}
as desired.

Case 3: If $3\le m\le n$ then Lemma \ref{recur} gives 
\begin{align*}
F_B(N)&\le \max(F_B((n-1)^2+m-2)+3(N-n)-2,F_B((n-1)^2+(m-1))+2N+n(n-3))\\
&\le \max(n^3+m^2-m+1,n^3+m^2)=n^3+m^2,
\end{align*}
as desired.

Case 4: If $m=n+1$ then Lemma \ref{recur} gives
\begin{align*}
F_B(N)&\le F_B((n-1)^2+m-2)+3(N-n)-2\\
&\le n^3+m^2-m+1=n^3+n^2+m(m-n),
\end{align*}
where $m=n+1$ is used in the final deduction.

Case 5: If $n+1<m\le 2n+1$ then $m - 2\ge (n - 1) + 1$ and Lemma \ref{recur} gives
\begin{align*}
F_B(N)&\le F_B((n-1)^2+m-2)+3(N-n)-2\\
&\le n^3+n^2+m^2-mn,
\end{align*}
as desired.

Thus the inductive step follows in all cases and the proof is complete.
\end{proof}

\section{Maximum Specialty over Forests and Planar Graphs}
Unlike the previous sections, where the methods have been largely combinatorial, the key method for these two results is clever summation by parts. The algebraic casting of the problem was the key observation for the solution given by the contestant on the original Team Selection Test problem and the specific use of summation by parts also appears in Brendan McKay's answer to \cite{273694} where specialty of planar graphs with a specific number of vertices rather than edges is maximized.
\begin{theorem}\label{main3}
The maximum specialty over all forests with $N$ edges is $1$ if $N = 1$ and $2N - 2$ if $N\ge 2$.
\end{theorem}
\begin{proof}
The case $N = 1$ is clear, so let $N\ge 2$. A specialty of $2N - 2$ is achieved with a path with $N$ edges.

Notice that given a forest, we can take two leaves in separate connected components and merge them. The resulting graph is still a forest, and the specialty has not decreased. Therefore, it suffices to study a graph $G$ that is a tree with $N$ edges and $N + 1$ vertices.

Let $d_1\ge\cdots\ge d_{N + 1}$ be the degrees of vertices $v_1, \ldots, v_{N + 1}$ of the graph $G$. Let $a_i$ for $1\le i\le N + 1$ be the number of edges between $v_i$ and all $v_j$ with $j < i$. Using summation by parts we obtain
\[S(G) = \sum_{i = 1}^{N + 1} a_id_i = \sum_{i = 1}^{N + 1} (d_i - d_{i + 1})\bigg(\sum_{j = 1}^i a_j\bigg),\]
where $d_{N + 2}$ is defined to be $0$. Now each difference $d_i - d_{i + 1}\ge 0$, and $\sum_{j = 1}^i a_j$ is the number of edges in the induced subgraph of $G$ obtained by restricting to vertices $v_1, \ldots, v_i$ only. Notice the corresponding subgraph has $i$ vertices and contains no cycles, so has at most $i - 1$ edges. Therefore $\sum_{j = 1}^i a_j\le i - 1$ for $1\le i\le N + 1$. Then
\[S(G)\le \sum_{i = 1}^{N + 1} (i - 1)(d_i - d_{i + 1}) = \sum_{i = 2}^{N + 1} d_i = 2N - d_1.\]
Since $N\ge 2$, and $G$ is a tree with $N$ edges, we know $d_1 + \cdots + d_{N + 1} = 2N$, which implies the last equality. Furthermore, this gives $d_1\ge\frac{2N}{N + 1} > 1$ so $d_1\ge 2$. Therefore $S(G)\le 2N - 2$, and the result follows.
\end{proof}

For the planar case, first notice that the graphs that $S(G)$ on $N$ edges, without any restrictions, are planar for $N\le 9$. Therefore it suffices to study $N\ge 10$. We now provide a inductive construction which provides the maximal specialty for $N\ge 33$.

Define $G_{N}$ as follows. First construct three vertices $v_1, v_2, v_3$ connected in a triangle and take $k = \left\lfloor\frac{N}{3}\right\rfloor + 2$. In each successive stage, add $v_i$ where $v_i$ connects to $v_{i - 1}, v_{i - 2}, v_{i - 3}$. Continue until we have $k$ vertices and $3k - 6$ edges total. This is clearly planar, since each new vertex can be added on the outside of the planar embedding of the graph we are constructing. Now, if $N\equiv 1\pmod{3}$ add an additional vertex $v_{k+1}$ that connects only to $v_k$. If $N\equiv 2\pmod{3}$ add an additional vertex $v_{k+1}$ that connects to $v_k, v_{k - 1}$. When $k\ge 6$ the graph on $\{v_1, \ldots, v_k\}$ has $k - 6$ vertices of degree $6$, and two vertices of each degree $3, 4, 5$. Indeed, $\deg v_1 = \deg v_k = 3, \deg v_2 = \deg v_{k - 1} = 4, \deg v_3 = \deg v_{k - 2} = 5$, and the rest have degree $6$. It is easy to check that for $k\ge 9$ the weights $3, 4, 5$ are each assigned to $6$ edges and the remaining edges have weight $6$. Thus the specialty of the graph on $\{v_1, \ldots, v_k\}$ with $N = 3k - 6$ edges is $6(N - 18) + 72 = 6N - 36$. Furthermore, adding $v_{k+1}$ in the case when $N\equiv 1\pmod{3}$ adds a total of $4$ to the specialty, and adding $v_{k+1}$ in the $N\equiv 2\pmod{3}$ case adds a total of $10$ to the specialty. Therefore the total specialty is $6(N - 1) - 36 + 4 = 6N - 38$ and $6(N - 2) - 36 + 10 = 6N - 38$ in these cases. Hence our construction, which is valid when $k\ge 9$ or $N\ge 21$, yields
\[S(G_N) = \bigg\{\begin{array}{lr}3N - 36, & \text{if } N\equiv 0\pmod{3}\\
6N - 38, & \text{otherwise}.\end{array}\]
\begin{figure}[h]
\begin{center}
  \includegraphics[width=.3\linewidth]{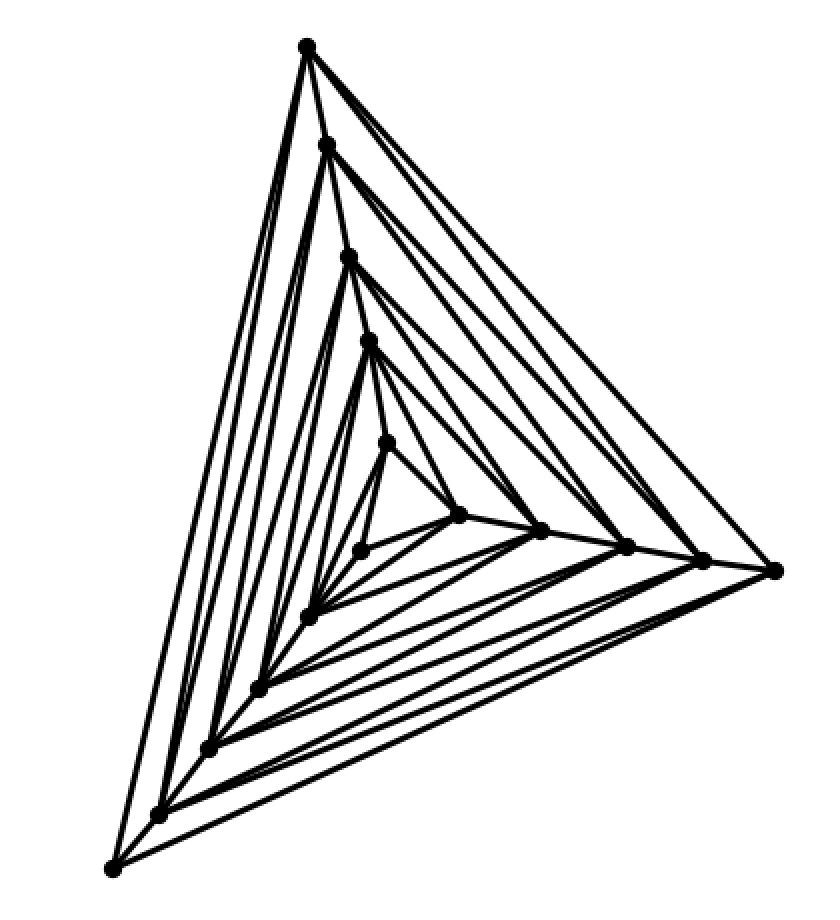}
  \end{center}
\caption{Maximal graph in Theorem \ref{main4} with $N = 42$ edges}
\label{fig:arckdef3}
\end{figure}
As we will see, this construction turns out to be optimal for the regime $N\ge 33$.
\begin{theorem}\label{main4}
The maximum specialty over all planar graphs with $N\ge 33$ edges is $6N-36$ if $N\equiv 0\pmod 3$ and $6N-38$ otherwise.
\end{theorem}
\begin{proof}
The constructions are given above. Now recall the classical fact that any planar graph with $n\ge 3$ vertices has at most $3n-6$ edges. We can exploit this inequality in the same vein as in the proof of Theorem \ref{main3}. Since $N\ge 33$ in this proof, the graphs we will consider always have at least $9$ vertices.

Let $v_1, \ldots, v_k$ be the vertices of $G$ with $N$ edges, and suppose that $d_1, \ldots, d_k$ are the respective degrees of the $v_i$ with $d_1\ge\cdots\ge d_k\ge 1$. Furthermore let $a_i$ for $1\le i\le k$ be the number of edges between $v_i$ and $v_j$ with $j < i$. Note that $\sum_{j = 1}^i a_j$ is the number of edges in the induced subgraph of $G$ obtained by restricting to vertices $v_1, \ldots, v_i$ only. Therefore $\sum_{j = 1}^i a_j\le 3i - 6$ for $i\ge 3$ and $a_1 = 0, a_1+a_2\le 1$. Furthermore, $\sum_{j = 1}^i a_j\le N$ for all $i$. Finally, for convenience let $d_{k + 1} = 0$. Then, as before,
\[S(G) = \sum_{i = 1}^k a_id_i = \sum_{i = 1}^k (d_i - d_{i + 1})\bigg(\sum_{j = 1}^i a_j\bigg).\]
We break into cases based on $N \pmod 3$.

Case 1: $N\equiv 0\pmod{3}$, in which case we find
\begin{align*}
S(G)&\le (d_2 - d_3) + \sum_{i = 3}^k (3i - 6)(d_i - d_{i + 1})\\
&\le d_2 + 2d_3 + 3\sum_{i = 4}^k d_i\\
&\le 6N - 3d_1 - 2d_2 - d_3,
\end{align*}
using $\sum_{i = 1}^k d_i = 2N$.

Case 2: $N\equiv 1\pmod{3}$, in which case we find
\begin{align*}
S(G)&\le (d_2 - d_3) + \sum_{i = 3}^{k - 1} (3i - 6)(d_i - d_{i + 1}) + (3k - 8)(d_k - d_{k + 1})\\
&\le d_2 + 2d_3 + 3\sum_{i = 4}^{k - 1} d_i + d_k\\
&\le 6N - 3d_1 - 2d_2 - d_3 - 2d_k\\
&\le 6N - 3d_1 - 2d_2 - d_3 - 2,
\end{align*}
where we used the fact that $\sum_{i=1}^k a_i = N$ and $\sum_{i=1}^k a_i\le 3k - 6$ to deduce $\sum_{i = 1}^k a_i\le 3k - 8$, based on the modular condition on $N$.

Case 3: $N\equiv 2\pmod{3}$, in which case we find
\begin{align*}
S(G)&\le (d_2 - d_3) + \sum_{i = 3}^{k - 1} (3i - 6)(d_i - d_{i + 1}) + (3k - 7)(d_k - d_{k + 1})\\
&\le d_2 + 2d_3 + 3\sum_{i = 4}^{k - 1} d_i + 2d_k\\
&\le 6N - 3d_1 - 2d_2 - d_3 - d_k\\
&\le 6N - 3d_1 - 2d_2 - d_3 - 1,
\end{align*}
where we used the fact that $\sum_{i=1}^{k}a_i = N$ and $\sum_{i=1}^{k}a_i\le 3k - 6$ to deduce $\sum_{i=1}^{k}a_i\le 3k - 7$, based on the modular condition on $N$. However, we can improve this bound by carefully considering the possible equality cases: we must have $d_k = a_k = 1$, hence $N = \sum_{i=1}^{k}a_i = \sum_{i=1}^{k-1}a_i + 1\le 3(k - 1) - 6 + 1 = 3k - 8$. Hence for any hypothetical equality cases $G$, we can sharpen to
\begin{align*}
S(G)&\le (d_2 - d_3) + \sum_{i = 3}^{k - 1} (3i - 6)(d_i - d_{i + 1}) + (3k - 8)(d_k - d_{k + 1})\\
&\le d_2 + 2d_3 + 3\sum_{i = 3}^{k - 2} d_i + d_k\\
&\le 6N - 3d_1 - 2d_2 - d_3 - 2d_k\\
&\le 6N - 3d_1 - 2d_2 - d_3 - 2
\end{align*}
and therefore there in fact are no equality cases, so that
\[S(G)\le 6N - 3d_1 - 2d_2 - d_3 - 2\]
for all $G$ with $N$ edges and $N \equiv 2\pmod3$.
\newline\newline
To finish note that if $G$ satisfies $3d_1 + 2d_2 + d_3\ge 36$, then we are done regardless of which case we are in. Therefore it suffices to consider $3d_1 + 2d_2 + d_3\le 35$. Thus since $d_3\ge 1$ we have $5d_2\le 3d_1 + 2d_2\le 34$ and hence $d_2\le 6$. Similarly since $6d_3\le 3d_1 + 2d_2 + d_3\le 35$ it follows that $d_3\le 5$, and thus $d_i\le d_3\le 5$ for $i\ge 3$. Now at most one edge has weight $d_2 = 6$, which is a potential edge between $v_1, v_2$. The remaining edges all have at least one vertex of degree at most $5$, so
\[S(G)\le 6 + 5(N - 1)\le 6N - 38\]
and the result follows for $N\ge 39$. For $33\le N\le 38$ a more careful analysis involving $d_k$ is necessary. Clearly $d_k\le 5$ since $G$ is planar, and we may assume we are working in the case where $3d_1 + 2d_2 + d_3 < 36$. In particular, it is still true that at most one edge has weight $6$ and the rest have weight $\le 5$. We can also assume the graph $G$ under consideration is a connected graph, since combining two vertices on the convex hulls of the embeddings of disconnected components $G_1, G_2$ yields a graph that is connected and has at least the same specialty as before.

Case 1: Suppose that $2\le d_k\le 3$. If $d_k=2$ then note that \[S(G)\le 6+2(2)+5(N-3)=5N-5\le 6N-38,\] which follows as $N\ge 33$. Otherwise $d_k=3$ and \[S(G)\le 6+3(3)+5(N-4)=5N-5\le 6N-38,\] which follows as $N\ge 33$.

Case 2: Suppose that $d_k=4$. Then note that
\[S(G)\le 6+4(4)+5(N-5)\le 6N-36,\]
which holds for $N\ge 33$. Thus the result is settled for $N\equiv 0\pmod{3}$. If $N\equiv 1\pmod{3}$, then we derived earlier the inequality
\[S(G)\le 6N-3d_1-2d_2-d_3-2d_k.\]
If $3d_1 + 2d_2 + d_3 + 2d_k\ge 38$ then it is settled. Otherwise, $3d_1 + 2d_2 + d_3 + 2d_k < 38$. Note that $30 = 38-2d_k>3d_1+2d_2+d_3\ge 6d_3$ so $d_3 < 5$. Also, $5d_2\le 3d_1 + 2d_2 < 30$ so $d_2\le 5$. Thus that all edges except perhaps between $v_1, v_2$ have weight $\le 4$ and
\[S(G)\le 4(N-1)+5<6N-38\]
for $N\ge 20$. Finally, if $N\equiv 2\pmod{3}$ then, similarly, we only need to handle the case when $3d_1+2d_2+d_3+d_k<38$. Then $6d_3<34$ or $d_3\le 5$ and $30\ge 38-d_3-d_k>3d_1+2d_2\ge 5d_2$ so $d_2 < 6$. Therefore it follows that
\[S(G)\le 5(N-4)+16\le 6N-38,\]
which holds for $N=35$ and $N=38$, the cases under consideration.

Case 3: Suppose that $d_k=5$. If $N\equiv 1 \pmod{3}$ then
\begin{align*}
S(G)&\le 6N-3d_1-2d_2-d_3-2d_k\\
&\le 6N-40<6N-38,
\end{align*}
as desired. If $N\equiv 2 \pmod{3}$ then
\[S(G)\le 6N-3d_1-2d_2-d_3-d_k,\]
so we only need consider the case $3d_1+2d_2+d_3+d_k<38$. That implies $3d_1<38-2d_2-d_3-d_k\le 18$, hence $d_1=5$. Therefore the degree sequence is $(5,\ldots,5)$, implying $N\equiv 0\pmod{5}$ and thus $N = 35$. Thus there are $14$ $5$'s in the degree sequence. However, by Theorem 1 in \cite{schmeichel1977planar}, no such planar graph exists. Finally, suppose $N\equiv 0\pmod{3}$. We need only consider cases when $3d_1+2d_2+d_3<36$. Thus $d_3 < 6$. Since $d_3\ge d_k = 5$, we conclude $d_3=5$. Now if $d_1=d_2=5$ then $N=\equiv 0\pmod{5}$ but we are only considering $N=33$ and $N=36$ in this case. If $d_1=6$ and $d_2=5$ the degree sequence is $(6,5,\ldots,5)$, hence $N = 33$ and there are $12$ $5$'s. However, no such planar graph exists by Theorem 2 in \cite{schmeichel1977planar}. Finally, if $d_1=d_2=6$ then note that the degree sequence is $(6,6,5,\ldots,5)$, hence $N=36$ is forced and there are $12$ $5$'s. In this case $S(G)\le 5(N-1)+6$, which is precisely $1$ more than the claimed bound of $6N-36$. For equality to occur, however, there must exist an edge between the two vertices of degree $6$. Removing this, we are left with a planar graph on $14$ vertices, each with degree $5$. However, this does not exist by Theorem 1 in \cite{schmeichel1977planar}. Therefore this case is complete.

Case 4: Finally, suppose that $d_k=1$. If $N\equiv 0\pmod{3}$ then $S(G)\le 6+1+5(N-2)$, which is at most the claimed bound of $6N-36$ when $N\ge 33$, as desired. If $N\equiv 2\pmod{3}$ then by the same reasoning above  $S(G)\le 6+5(N-2)+1$ and this is at most the claimed bound of $6N-38$ for $N\ge 35$, as desired. Finally if $N\equiv 1 \pmod{3}$ then note that $S(G)\le 1+6+5(N-2)$ as above and this is less than $6N-38$ for $N=37$ and exactly one more than the claimed bound for $N=34$. Consider $N = 34$. In order to violate the claimed bound, equality must hold. In this case, we see $d_{k-1}\notin\{2,3,4\}$. In particular, in any of these cases there is an edge from $v_{k-1}$ not connected to $v_k$ and the above bound estimates that this edge has weight $5$ while it does not. If $d_{k-1}=1$ then $v_{k-1}$ and $v_k$ are forced to connect in the equality case. However, this creates a disconnected component in $G$, which we assumed was not the case. Thus $d_{k-1}\ge 5$. Remembering $3d_1+2d_2+d_3<36$, we have $d_3\le 5$ and thus $d_3=\cdots=d_{k-1}=5$. Furthermore, for this equality to hold we need the edge of weight $6$ between $v_1, v_2$. Thus $d_2=6$ and therefore since $3d_1+2d_2+d_3< 36$ we find $3d_1<19$, implying $d_1=6$. Finally, $G$ has degree sequence $(6,6,5,\ldots,5,1)$ with $11$ $5$'s. For equality to occur we have an edge between the two vertices of degree $6$ and removing this vertex gives $(5,\ldots,5,1)$ where there are $13$ $5$'s. Removing the vertex of degree $1$ leaves a planar graph $(5,\ldots,5,4)$ with $12$ $5$'s. By Theorem 2 in \cite{schmeichel1977planar}, such a planar graph does not exist and we are finished.
\end{proof}
To see that the threshold above is sharp, notice that the icosahedral graph $I$ is planar, $5$-regular, has $N = 30$ edges, and satisfies $S(I) = 5\cdot 30 > 6\cdot 30 - 36$. Select arbitrary edge $uv$ of $I$ and add a new vertex attached only to $u, v$ to create planar graph $I'$ with $N = 32$ edges. Delete edge $uv$ from $I'$ to create planar graph $I''$ with $N = 31$ edges. Then $S(I') = 6\cdot 1 + 5\cdot 29 + 2\cdot 2 > 6\cdot 32 - 38$ and $S(I'') = 5\cdot 29 + 2\cdot 2 > 6\cdot 31 - 38$. Therefore, the result above does not hold for $N = 30, 31, 32$. With little difficulty one can adapt the proof above to show that $I$ is optimal for $N=30$. Indeed, we have
\[S(G)\le 6\cdot 30 - 3d_1 - 2d_2 - d_1\le 150\]
unless $3d_1 + 2d_2 + d_3 < 30$, in which case $d_2\le 5$ follows. Hence all edges have weight $\le 5$, implying $S(G)\le 5\cdot 30 = 150$. We suspect that similar arguments will yield that either $I''$ and $I'$ or graphs with very similar degree sequences will be optimal for $N = 31, 32$, respectively.

We end by noting that, similar to a comment by Brendan McKay in \cite{273694}, if a graph $G$ satisfies the property that every subgraph has average degree at most $\Delta$, then its specialty satisfies $S(G)\le\Delta N$. The proof uses summation by parts in a identical manner to the above proofs. In particular, any graph family closed under minors, other than the set of all graphs, by \cite{mader1967homomorphieeigenschaften}, has linear specialty.

\section{Open Questions}
Given the results of Theorem \ref{main4}, we immediately ask the following question.
\begin{question}
What is the maximum specialty of a planar graph when restricted to $N$ edges with $N$ between $10$ and $32$ but not equal to $30$ edges? 
\end{question}
The results of this paper however otherwise settle the maximum specialty of a graph when restricted to a specific number of edges in the case of all graphs, bipartite graphs, forests, and planar graphs, and therefore it is natural to ask for finer control of specialty. In particular it is natural to ask which graphs maximize specialty with a fixed number of vertices and edges.  However note that the optimizing graphs in Theorems \ref{main}, \ref{main2}, \ref{main3}, \ref{main4} always have the minimum possible number of vertices and therefore one can simply add on isolated vertices until the required vertex count. Therefore it is necessary for one to further restrict to the case where $G$ is connected. Let $\mathcal{CG}(N, n)$ be the set of connected graphs with $N$ edges and $n$ vertices.
\begin{question}
What is the behaviour of 
\[F(N,n) = \max_{G\in\mathcal{CG}(N, n)} (S(G))?\]
In particular, how does the behaviour change when $N$ grows linearly in $n$ versus when $N$ grows quadratically in $n$? How does this behavior change when $G$ is further restricted to be bipartite?
\end{question}
\section{Acknowledgements}
The authors would like to thank Evan Chen for introducing them to the Question 1 and suggesting the associated general problem. The authors would also like Evan Chen and Colin Defant for reading the manuscript and providing helpful comments.

\bibliographystyle{plain}
\bibliography{name}

\end{document}